\documentclass{article}

\usepackage{amsmath}
\usepackage{amssymb}
\usepackage{graphicx}
\usepackage{tikz}
\usetikzlibrary{arrows.meta,calc}
\usepackage{bm}
\usepackage{enumitem}
\usepackage{mathdots}
\usepackage{booktabs}
\usepackage{rotating}
\usepackage{listings}
\usepackage{hyperref}
\usepackage{xcolor}
\definecolor{Mycolor}{RGB}{10,74,38}
\hypersetup{
  colorlinks=true,
  linkcolor=Mycolor,
  citecolor=Mycolor,
  urlcolor=Mycolor
}

\usepackage[ruled,linesnumbered,longend]{algorithm2e}
\usepackage[a4paper,left=2.8cm,right=2.8cm,top=2.5cm,bottom=2.5cm]{geometry}
\usepackage{fancyhdr}
\usepackage[framemethod=tikz]{mdframed}
\newtheorem{theorem}{Theorem}[section]

\newtheorem{lemma}{Lemma}[section]
\newtheorem{definition}{Definition}[section]

\makeatletter 
\@addtoreset{equation}{section}
\makeatother  

\newtheorem{remark}{Remark}[section]

\newenvironment{proof}{{\noindent\it Proof.}\quad}{\hfill $\square$\\}

\newcommand{\Sd}{\mathbb S^d}
\newcommand{\R}{\mathbb R}
\newcommand{\dd}{\,\mathrm d}
\newcommand{\norm}[1]{\left\lVert #1\right\rVert}
\newcommand{\abs}[1]{\left\lvert #1\right\rvert}

\begin{document}

\title{Nested QMC designs on spheres}

\author{ Hao-Ning Wu\footnotemark[2] \qquad Xiaosheng Zhuang\footnotemark[3]
       }

\renewcommand{\thefootnote}{\fnsymbol{footnote}}
\footnotetext[2]{School of Mathematical Sciences, Xiamen University, Xiamen, Fujian, 361005, China (hnwu@xmu.edu.cn)}
\footnotetext[3]{Department of Mathematics, City University of Hong Kong, Hong Kong, China (xzhuang7@cityu.edu.hk)}

\maketitle

\begin{abstract}
Nested cubature rules, in which each refinement retains all previously used nodes and thus reuses earlier function evaluations, are natural in multilevel and adaptive integration. For every fixed $s>d/2$, we show that the equal-weight QMC integration rate on $\mathbb S^d$ is compatible with such nested point sets. In the subcritical range $d/2<s<d$, cumulative unions of geometrically growing QMC blocks yield nested QMC design sequences whose successive cardinality ratios converge to any prescribed $\rho>1$. At and above the critical index $s=d$, where the block-averaging estimate no longer yields the optimal rate, we prove an equal-weight completion theorem based on low-frequency discrepancy cancellation. A block-sensitive estimate sharpens the iteration and yields nested QMC designs for all $s\ge d$ with $N_{j+1}\lesssim(j+1)^{2s/d-1}N_j$. Every prescribed finite point set also admits an optimal-rate completion.
\end{abstract}



\section{Introduction}

Numerical integration on the unit sphere $\Sd\subset\R^{d+1}$ is a basic
ingredient in approximation theory and numerical analysis.  For $s>d/2$, the optimal worst-case error attainable by $N$-point
cubature rules on $H^s(\Sd)$ has order $N^{-s/d}$; see, e.g.,
\cite{BrauchartHesse2007,Hesse2006Lower}.  Classical constructions
attaining this rate typically rely on polynomial exactness.  Quasi-Monte Carlo (QMC) designs
were introduced in \cite{BrauchartSaffSloanWomersley2014} to retain the
optimal integration rate without imposing polynomial exactness.
In this paper, we investigate whether the same rate is
compatible with \emph{nested} point sets.

\subsection{QMC designs and spherical designs}\label{sec:qmc-background}

Let $\sigma_d$ be the normalized surface measure on $\Sd$. For \(f\in L^2(\mathbb S^d)\), we consider the spherical integral
\[\mathcal I(f):=\int_{\Sd}f\,\dd\sigma_d.\]  
For $\ell\geq0$, let
$\mathcal H_\ell(\Sd)$ be the space of spherical harmonics of degree $\ell$,
with dimension $Z(d,\ell)$, and choose a real orthonormal basis
$\{Y_{\ell,k}\}_{k=1}^{Z(d,\ell)}$.  The corresponding eigenvalue of
$-\Delta_{\Sd}$ is $\lambda_\ell=\ell(\ell+d-1)$.  With Fourier--Laplace coefficients
 $\widehat f_{\ell,k}:=\int_{\Sd}fY_{\ell,k}\,\dd\sigma_d$, the Sobolev space \(H^s(\Sd)\) for \(s\ge0\) is equipped with the
spectral norm
\begin{equation}\label{eq:sobolev-norm}
 \norm{f}_{H^s(\Sd)}^2
 :=\sum_{\ell=0}^\infty(1+\lambda_\ell)^s
 \sum_{k=1}^{Z(d,\ell)}|\widehat f_{\ell,k}|^2.
\end{equation}
When $s>d/2$, the
embedding $H^s(\Sd)\hookrightarrow C(\Sd)$ makes point evaluation continuous.
Thus, for a finite set $\boldsymbol{X}\subset\Sd$, the worst-case error
\begin{equation*}
 \mathcal E_s(\boldsymbol{X}):=
 \sup_{\norm{f}_{H^s(\Sd)}\leq1}
 \abs{\frac1{|\boldsymbol{X}|}\sum_{x\in \boldsymbol{X}}f(x)-\mathcal I(f)}
\end{equation*}
of an equal-weight cubature rule 
is well defined.

\begin{definition}[QMC design sequence \cite{BrauchartSaffSloanWomersley2014}]\label{def:qmc-design}
 Fix $s>d/2$, and let $\mathcal N\subset\mathbb N$ be infinite.  A family
 $\{\boldsymbol{X}_N\}_{N\in\mathcal N}$ of point sets in $\Sd$, with
 $|\boldsymbol{X}_N|=N$, is an \emph{equal-weight QMC design
 sequence for $H^s(\Sd)$} if there exists a constant $C_{s,d}>0$,
 independent of $N$, such that
 \begin{equation}\label{eq:qmc-rate}
  \sup_{\norm{f}_{H^s(\Sd)}\leq1}
  \abs{
   \frac1N\sum_{x\in \boldsymbol{X}_N}f(x)-\int_{\Sd}f\,\dd\sigma_d
  }
  \leq C_{s,d}N^{-s/d},
  \qquad N\in\mathcal N.
 \end{equation}
 Its \emph{QMC strength} is the supremum of the values of $s$ for which
 \eqref{eq:qmc-rate} holds.  A sequence of infinite strength is called
 \emph{generic}.
\end{definition}

This is the original equal-weight definition of QMC designs in
\cite{BrauchartSaffSloanWomersley2014}, with related discrepancy--energy
formulations studied in \cite{BilykDaiMatzke2018}.  All asymptotic statements in $N$ are
understood as $N\to\infty$ through $\mathcal N$.  Here ``sequence'' means a
family indexed by an arbitrary infinite set of cardinalities.
The principal examples of QMC designs are spherical $t$-designs \cite{DelsarteGoethalsSeidel1977}.  Let $\Pi_t(\Sd)$ be the space of spherical polynomials of degree at most $t$. The point set $\boldsymbol{X}_{N,t}\subset\Sd$ is said to be a spherical
$t$-design when
\begin{equation*}
 \frac1{|\boldsymbol{X}_{N,t}|}\sum_{x\in \boldsymbol{X}_{N,t}}p(x)
 =\int_{\Sd}p\,\dd\sigma_d
 \qquad \forall\, p\in\Pi_t(\Sd).
\end{equation*}
Every spherical \(t\)-design satisfies the lower bound
\(|\boldsymbol{X}_{N,t}|\gtrsim t^d\); see \cite{DelsarteGoethalsSeidel1977}.
Moreover, optimal-order spherical \(t\)-designs with
\(|\boldsymbol{X}_{N,t}|\asymp t^d\) exist, as established in \cite{BondarenkoRadchenkoViazovska2013}. Hence, for every \(s>d/2\), an optimal-order family of spherical designs satisfies
\begin{equation}\label{eq:designs-are-qmc}
 \sup_{\norm{f}_{H^s(\Sd)}\leq1}
 \abs{
  \frac1{|\boldsymbol{X}_{N,t}|}\sum_{x\in \boldsymbol{X}_{N,t}}f(x)
  -\int_{\Sd}f\,\dd\sigma_d
 }
 \lesssim t^{-s}\asymp |\boldsymbol{X}_{N,t}|^{-s/d},
\end{equation}
where the first estimate was proved on $\mathbb S^2$ in
\cite{HesseSloan2005WCE,HesseSloan2006ArbitraryOrder} and further in arbitrary
dimension in \cite{BrauchartHesse2007,HesseMhaskarSloan2007}.  Consequently, optimal-order spherical designs are generic equal-weight QMC designs.  Polynomial
exactness is therefore a sufficient mechanism for attaining the QMC rate \eqref{eq:qmc-rate},
but it is not part of Definition~\ref{def:qmc-design}.

\subsection{Why nesting?}\label{sec:nested-motivation}

In multilevel and adaptive computations, cubature rules are refined for a
fixed integrand.  Nested nodes preserve previously computed function values,
so each refinement requires evaluations only at the newly added nodes.  The
same reuse principle underlies
recent multiscale constructions based on nested sampling measures; see, for
example, \cite{WuZhuang2026NestedMZ}.  This motivates the condition
$\boldsymbol{X}_{N_j}\subset \boldsymbol{X}_{N_{j+1}}$,
where $N_j,\,N_{j+1}\in\mathcal{N}$ are adjacent indices in $\mathcal{N}$.

Spherical designs provide an exact counterpart of this nesting condition. Let $\boldsymbol{Y}_M$ be a spherical
$t_1$-design with $t_1<t$.  Zheng and Zhuang proved in \cite{ZhengZhuang2024Nested}
 that $\boldsymbol{Y}_M$ can be enlarged
by a set $\boldsymbol{Z}_N$ so that
$\boldsymbol{Y}_M\cup\boldsymbol{Z}_N$ is a spherical $t$-design; their
theorem in fact allows $\boldsymbol{Y}_M$ to be an arbitrary prescribed point
set.  If $M\asymp t_1^d$, their general bound for this two-level completion is $N+M=O(t^{2d+1})$.
Here, $t_1$ is the degree of the inherited design, whereas $t$ is the target
degree after completion.  They attained the optimal orders $t_1^d$ and $t^d$
when $t=mt_1$ for a fixed rational $m>1$, with constants depending on $m$,
and conjectured the same conclusion uniformly for all $t_1<t$, with constants
depending only on $d$.

This naturally leads us to ask whether nesting is compatible with the optimal Sobolev integration rate without imposing polynomial exactness. QMC designs provide the appropriate framework for this question, and we further investigate the attainable growth of the successive cardinalities.

\subsection{Equal-weight nested QMC designs}
\label{sec:nested-definition}

For QMC designs, there is no intrinsic degree parameter like $t$ in spherical $t$-designs.  Let
$\mathcal N=\{N_j:j\geq0\}\subset\mathbb N$ with
$N_0<N_1<\cdots$, and set
\begin{equation*}
 \mathcal Q(\boldsymbol{X}_{N_j};f):=
 \frac1{N_j}\sum_{x\in \boldsymbol{X}_{N_j}}f(x).
\end{equation*}

\begin{definition}[Equal-weight nested QMC design sequence]
\label{def:nested-qmc}
 Fix $s>d/2$.  A sequence $\{\boldsymbol{X}_{N_j}\}_{j\geq0}$ with
 $|\boldsymbol{X}_{N_j}|=N_j$ is an \emph{equal-weight nested QMC design sequence for
 $H^s(\Sd)$} if $\boldsymbol{X}_{N_j}\subset \boldsymbol{X}_{N_{j+1}}$
 and
 \begin{equation*}
  \sup_{\norm{f}_{H^s(\Sd)}\leq1}
  \abs{
   \mathcal Q(\boldsymbol{X}_{N_j};f)-\mathcal I(f)
  }
  \leq C_{s,d}N_j^{-s/d}.
 \end{equation*}
\end{definition}

As in the original QMC definition, no relation between successive cardinalities is imposed. The main conclusions take different forms below and at or above the critical
index $s=d$.  They are collected in the following theorem.

\begin{theorem}[Nested QMC designs]
\label{thm:main-results}
 Let $d\geq2$.
 \begin{enumerate}
  \item[(i)] If $d/2<s<d$, then for every $\rho>1$ there is an equal-weight
  nested QMC design sequence $\{\boldsymbol{X}_{N_j}\}_{j\geq0}$ for $H^s(\Sd)$ such
  that
  \begin{equation*}
   \mathcal E_s(\boldsymbol{X}_{N_j})\lesssim N_j^{-s/d},
  \end{equation*}
  where ${N_{j+1}}/{N_j}\longrightarrow\rho\text{ as }j\rightarrow\infty$.
  \item[(ii)] If $s\geq d$, there is an equal-weight nested QMC design sequence for
  $H^s(\Sd)$ such that
  \begin{equation*}
   \mathcal E_s(\boldsymbol{X}_{N_j})\lesssim N_j^{-s/d},
  \end{equation*}
  where $N_{j+1}\lesssim (j+1)^{2s/d-1}N_j$. Moreover, every prescribed $N$-point set admits an $N'$-point completion satisfying
  \begin{equation*}
      \mathcal E_s(\boldsymbol{X}_{N'})\lesssim (N')^{-s/d},
  \end{equation*}
  where $N'\lesssim N^{1+2s/d}$.
   \end{enumerate}
\end{theorem}

Both parts of Theorem~\ref{thm:main-results} attain the optimal order
\(N_j^{-s/d}\), but by different constructions.
Part~\textup{(i)} uses cumulative unions of geometric QMC blocks. Part~\textup{(ii)} instead cancels the low-frequency discrepancy while
preserving the block structure.

\subsection{Notations}
For $x,y\in\Sd$, let $\operatorname{dist}(x,y):=\arccos(x\cdot y)$ be the geodesic distance and 
$B(x,\theta):=\{y\in\Sd:\operatorname{dist}(x,y)<\theta\}$ the open geodesic ball.
For a finite point set $\boldsymbol{A}\subset\Sd$, its separation is
\begin{equation*}
 q(\boldsymbol{A}):=
 \min_{{x,y\in\boldsymbol{A},\,x\ne y}}
 \operatorname{dist}(x,y),
\end{equation*}
with the convention $q(\boldsymbol{A})=\infty$ when
$|\boldsymbol{A}|=1$.  For $\delta>0$, we call $\boldsymbol{A}$
$\delta$-separated if $q(\boldsymbol{A})\geq\delta$.
An $N$-point set $\boldsymbol{X}_N$ is said to admit a
$c$-separated $K$-block decomposition if it can be written as a disjoint
union of nonempty sets,
\begin{equation}\label{eq:block-separation}
 \boldsymbol{X}_N=\mathop{\dot\bigcup}_{\nu=1}^K\boldsymbol{A}_\nu,
 \qquad
 q(\boldsymbol{A}_\nu)\geq cN^{-1/d},
 \quad 1\leq\nu\leq K.
\end{equation}
Every $N$-point set has such a decomposition with $K=N$, obtained by taking
singleton blocks; smaller values of $K$ measure how efficiently the set can
be partitioned into uniformly separated subsets.
We shall also use the unnormalized discrepancy measure
\begin{equation*}
 \mathcal D_{\boldsymbol{X}}
 :=\sum_{x\in\boldsymbol{X}}\delta_x-|\boldsymbol{X}|\sigma_d
\end{equation*}
of a finite point set $\boldsymbol{X}\subset\Sd$.  Here $H^{-r}(\Sd)$
denotes the Hilbert-space dual of $H^r(\Sd)$, equipped with the dual norm
$\norm{\mu}_{H^{-r}(\Sd)}
 :=\sup_{\norm{f}_{H^r(\Sd)}\leq1}
 \abs{\langle\mu,f\rangle}$.
For the spectral norm \eqref{eq:sobolev-norm}, this is equivalently
\begin{equation}\label{eq:negative-sobolev-norm}
 \norm{\mu}_{H^{-r}(\Sd)}^2
 :=\sum_{\ell=0}^\infty(1+\lambda_\ell)^{-r}
 \sum_{k=1}^{Z(d,\ell)}
 \abs{\langle\mu,Y_{\ell,k}\rangle}^2.
\end{equation}
For every $r>d/2$,
$\mathcal D_{\boldsymbol{X}}$ belongs to $H^{-r}(\Sd)$, and the error
functional of the equal-weight rule on $\boldsymbol{X}$ is represented by
$|\boldsymbol{X}|^{-1}\mathcal D_{\boldsymbol{X}}$.  In particular, $\mathcal E_r(\boldsymbol{X})
 =|\boldsymbol{X}|^{-1}
  \norm{\mathcal D_{\boldsymbol{X}}}_{H^{-r}(\Sd)}$.

Throughout, $A\lesssim B$ means that $A\leq CB$ with a constant $C$
independent of the asymptotic parameters, and $A\asymp B$ means that both
$A\lesssim B$ and $B\lesssim A$ hold.  Dependence on fixed parameters is
indicated by a subscript when it is relevant.

\section{The subcritical range $d/2<s<d$}
\label{sec:subcritical-range}

For a disjoint union, the integration-error functional is the
cardinality-weighted sum of the block error functionals.  When $d/2<s<d$,
an $m$-point QMC block contributes at most $m^{1-s/d}$ before normalization,
so a geometric sequence of such contributions is controlled by its final
term.  This yields the following result.

\begin{theorem}[Subcritical nested QMC designs]
\label{thm:subcritical-block}
 Given $d\geq2$, $d/2<s<d$, and $\rho>1$, there exist pairwise disjoint point
 sets $\boldsymbol{B}_j\subset\Sd$, with $|\boldsymbol{B}_j|=m_j:=\lceil\rho^j\rceil$,
 such that, for
 \begin{equation*}
  \boldsymbol{X}_{N_J}:=\bigcup_{j=0}^J\boldsymbol{B}_j,
  \qquad
  N_J:=|\boldsymbol{X}_{N_J}|=\sum_{j=0}^Jm_j,
 \end{equation*}
 the sequence $\{\boldsymbol{X}_{N_J}\}_{J\geq0}$ is an equal-weight nested
 QMC design sequence for $H^s(\Sd)$ in the sense of
 \begin{equation}\label{eq:subcritical-block-bound}
  \sup_{\norm{f}_{H^s(\Sd)}\leq1}
  \abs{
   \frac1{N_J}\sum_{x\in \boldsymbol{X}_{N_J}}f(x)
   -\int_{\Sd}f\,\dd\sigma_d
  }
  \lesssim N_J^{-s/d}.
 \end{equation}
 Moreover, $ {N_{J+1}}/{N_J}\longrightarrow\rho$ as $J\rightarrow \infty$.
\end{theorem}

\begin{proof}
 The existence theory for QMC designs in
 \cite{BrauchartSaffSloanWomersley2014} provides, for every sufficiently
 large $m$, an $m$-point set $\boldsymbol{B}\subset\Sd$ satisfying
 \begin{equation}\label{eq:subcritical-block-qmc}
  \sup_{\norm{f}_{H^s(\Sd)}\leq1}
  \abs{
   \frac1m\sum_{x\in\boldsymbol{B}}f(x)
   -\int_{\Sd}f\,\dd\sigma_d
  }
  \leq C_{s,d}m^{-s/d}.
 \end{equation}
 The constant is independent of $m$.  Since $m_j\to\infty$, this gives the
 required blocks for all but finitely many $j$.  Let $J_0$ be the set of
 remaining indices and choose an arbitrary $m_j$-point set
 $\boldsymbol{B}_j$ for each $j\in J_0$.  Since $s>d/2$, point evaluation is
 continuous and hence $\mathcal E_s(\boldsymbol{B}_j)<\infty$.  Replacing
 $C_{s,d}$ by $\max\left\{C_{s,d},\; 
  \max_{j\in J_0}m_j^{s/d}\mathcal E_s(\boldsymbol{B}_j)\right\}$
 makes \eqref{eq:subcritical-block-qmc} valid for every $j\geq0$.

 To obtain the nested sets by cumulative union while preserving
 $N_J=\sum_{j=0}^Jm_j$, we arrange the blocks to be pairwise disjoint, so
 that every block contributes $m_j$ new nodes.  Both $\sigma_d$ and the
 $H^s(\Sd)$-norm are rotation invariant, implying $\mathcal E_s(R\boldsymbol{B}_j)=\mathcal E_s(\boldsymbol{B}_j)$
 for $R\in SO(d+1)$, where $SO(d+1)$ denotes the rotation group of
 $\R^{d+1}$.
 Suppose that $\boldsymbol{B}_0,\ldots,\boldsymbol{B}_{j-1}$ have already
 been made pairwise disjoint, and let
  $\boldsymbol{F}_{j-1}:=
  \bigcup_{k=0}^{j-1}\boldsymbol{B}_k$.
 This is a finite set.  Let $\mu$ be the normalized Haar measure on $SO(d+1)$.
 For fixed $x\in\Sd$, the probability measure $A\mapsto\mu(\{R\in SO(d+1):Rx\in A\})$
 is rotation invariant and hence equals $\sigma_d$.  Therefore, for every
 $x,y\in\Sd$,
  $\mu(\{R\in SO(d+1):Rx=y\})=\sigma_d(\{y\})=0$.
 Hence, the exceptional set
 \begin{equation*}
  \{R\in SO(d+1):R\boldsymbol{B}_j\cap
       \boldsymbol{F}_{j-1}\ne\varnothing\}
  =
  \bigcup_{{x\in\boldsymbol{B}_j,\,y\in\boldsymbol{F}_{j-1}}}
  \{R\in SO(d+1):Rx=y\}
 \end{equation*}
 is a finite union of Haar-null sets and therefore cannot exhaust
 $SO(d+1)$.  We may thus rotate $\boldsymbol{B}_j$ away from all preceding
 blocks without changing its QMC error.  Starting with $\boldsymbol{B}_0$
 and repeating this choice at each stage $j$ yields a sequence of pairwise
 disjoint blocks with the original QMC bounds.

 For every $f\in H^s(\Sd)$, the equal-weight error of the union decomposes as
 \begin{equation*}
  \frac1{N_J}\sum_{x\in \boldsymbol{X}_{N_J}}f(x)
  -\int_{\Sd}f\,\dd\sigma_d
  =\sum_{j=0}^J\frac{m_j}{N_J}
   \left(
    \frac1{m_j}\sum_{x\in \boldsymbol{B}_j}f(x)
    -\int_{\Sd}f\,\dd\sigma_d
   \right),
 \end{equation*}
 implying
 \begin{equation}\label{equ:sec2estimate}
\begin{split}
  &\sup_{\norm{f}_{H^s(\Sd)}\leq1}
  \abs{
   \frac1{N_J}\sum_{x\in \boldsymbol{X}_{N_J}}f(x)
   -\int_{\Sd}f\,\dd\sigma_d
  }\\
  &\leq \sum_{j=0}^J\frac{m_j}{N_J}
    \sup_{\norm{f}_{H^s(\Sd)}\leq1}
    \left|
     \frac1{m_j}\sum_{x\in \boldsymbol{B}_j}f(x)
     -\int_{\Sd}f\,\dd\sigma_d
    \right|
  \lesssim \frac1{N_J}\sum_{j=0}^Jm_j^{1-s/d}.
  \end{split}
 \end{equation}
 Since $\lceil\rho^j\rceil=\rho^j+O(1)$,
 \begin{equation}\label{eq:subcritical-cardinality}
  N_J=\frac{\rho^{J+1}-1}{\rho-1}+O(J+1)
  \asymp_\rho \rho^{J+1}.
 \end{equation}
 In particular, $N_J\asymp_\rho m_J$.
 Let $\alpha:=1-s/d>0$.  Since $m_j\asymp\rho^j$, the geometric-series
 estimate gives
 \begin{equation*}
  \sum_{j=0}^Jm_j^\alpha
  \lesssim_\alpha \sum_{j=0}^J\rho^{\alpha j}
  =\rho^{\alpha J}\sum_{k=0}^J\rho^{-\alpha k}
  \leq\frac{\rho^{\alpha J}}{1-\rho^{-\alpha}}
  \lesssim_\alpha
  \frac{m_J^\alpha}{1-\rho^{-\alpha}}
  \asymp_{\rho,\alpha}N_J^\alpha.
 \end{equation*}
 Dividing by $N_J$ proves \eqref{eq:subcritical-block-bound}.  Finally,
 \eqref{eq:subcritical-cardinality} gives
 $N_{J+1}/N_J\to\rho$.
\end{proof}

\section{The critical and supercritical ranges $s\geq d$}
\label{sec:critical-supercritical-range}

The cumulative-block argument of Section~\ref{sec:subcritical-range} has a
threshold at $s=d$.  Indeed, for geometrically increasing QMC blocks with
$m_j\asymp\rho^j$ points, the estimate \eqref{equ:sec2estimate} gives
\begin{equation*}
 \mathcal E_s(\boldsymbol{X}_{N_J})
 \lesssim \frac1{N_J}\sum_{j=0}^Jm_j^{1-s/d}.
\end{equation*}
If $s=d$, then $m_j^{1-s/d}=1$, so the sum equals $J+1\asymp\log N_J$.
If $s>d$, since $m_j\asymp\rho^j$,
\begin{equation*}
 \sum_{j=0}^Jm_j^{1-s/d}
 \lesssim \sum_{j=0}^J\rho^{-(s/d-1) j}
 \leq\frac1{1-\rho^{-(s/d-1)}},
\end{equation*}
uniformly in $J$.  Dividing by $N_J$ and denoting by $N=N_J$, we obtain
\begin{equation*}
 \mathcal E_s(\boldsymbol{X}_N)
 \lesssim
 \begin{cases}
  N^{-1}\log N,&s=d,\\
  N^{-1},&s>d.
 \end{cases}
\end{equation*}
At $s=d$, the target is $N^{-1}$, so the direct
argument loses only a logarithmic factor.  For $s>d$, the target is
$N^{-s/d}$, and the direct argument misses it by the
factor $N^{s/d-1}$.  This limitation of the block-averaging estimate comes from the earliest blocks:
their unnormalized discrepancy remains of constant order and is reduced
only by the final normalization.  We therefore replace passive block
accumulation by a completion step that compensates the low-frequency
discrepancy of the inherited nodes, leading to another approach establishing the desired error rate of nested QMC designs.

\subsection{Completion controlled by block complexity}
\label{sec:block-completion}

The size of the correcting block is controlled by the separated-block
structure of the prescribed set.  We begin with the equal-weight cubature result
used in its construction.

\begin{lemma}[Equal-weight cubature for comparable measures]
\label{lem:comparable-measure-block}
 Let $w$ be measurable on $\Sd$ and satisfy
 \begin{equation}\label{eq:comparable-density}
  \int_{\Sd}w\,\dd\sigma_d=1,
  \qquad \frac12\leq w\leq\frac32.
 \end{equation}
 For every sufficiently large integer $m$, there is an $m$-point set
 $\boldsymbol{B}_m\subset\Sd$ such that, for every compact interval
\(
[r_0,S]\subset (d/2,\infty),
\)
the estimate
 \begin{equation}\label{eq:comparable-measure-qmc}
  \sup_{\norm{f}_{H^r(\Sd)}\leq1}
  \abs{
   \frac1m\sum_{x\in \boldsymbol{B}_m}f(x)
   -\int_{\Sd}fw\,\dd\sigma_d
  }
  \leq C_{r_0,S,d}m^{-r/d}
 \end{equation}
 holds simultaneously for all \(r\in[r_0,S]\).
 The points may be chosen with separation at least $c_dm^{-1/d}$.  The
 constants are independent of $w$ and $m$. 
\end{lemma}

\begin{proof}
 By \eqref{eq:comparable-density} and the volume growth of spherical caps, we have $\int_{B(x,\theta)}w\,\dd\sigma_d\asymp\theta^d$
 for $x\in\Sd$ and $0<\theta\leq\pi$,
 with constants independent of $w$.  Thus $w$ is a doubling weight with a
 uniformly bounded doubling constant and $\|w\|_{L^\infty}\leq3/2$.
 The hypothesis on the number of nodes $m$ in Dai and Feng's theorem
 \cite[Theorem~1.1]{DaiFeng2019} is therefore satisfied whenever
 $m\geq C_dn^d$.  Choosing $n\asymp m^{1/d}$, the theorem then provides $m$ distinct nodes $\boldsymbol{B}_m$ forming a strict
 Chebyshev-type cubature formula of degree $n$:
 \begin{equation}\label{eq:comparable-measure-exactness}
  \frac1m\sum_{x\in \boldsymbol{B}_m}p(x)
  =\int_{\Sd}pw\,\dd\sigma_d,
  \qquad p\in\Pi_n(\Sd),
 \end{equation}
 with $q(\boldsymbol{B}_m)\geq c_dm^{-1/d}$.  We then define the error
 functional
 \begin{equation}\label{equ:errorfunctional2}
  \mathcal L_m(f):=
  \frac1m\sum_{x\in\boldsymbol{B}_m}f(x)
  -\int_{\Sd}fw\,\dd\sigma_d.
 \end{equation}
 Let $\{\eta_k\}_{k\geq0}$ be a smooth dyadic partition of unity on
 $[0,\infty)$, with $\operatorname{supp}\eta_0\subset[0,2]$ and
 $\operatorname{supp}\eta_k\subset[2^{k-1},2^{k+1}]$ for $k\geq1$.
 We define the $k$th spectral band by
 \begin{equation*}
  F_k(x):=
  \sum_{\ell=0}^\infty\eta_k(\ell)
  \sum_{j=1}^{Z(d,\ell)}
  \left(\int_{\Sd}f(y)Y_{\ell,j}(y)\,\dd\sigma_d(y)\right)
  Y_{\ell,j}(x).
 \end{equation*}
 In particular, $F_k$ has degree at most $2^{k+1}$, and
 $f=\sum_{k\geq0}F_k$ in $H^r(\Sd)$.  The finite overlap of the multipliers
 and the spectral definition of the Sobolev norm give
 \begin{equation}\label{eq:dyadic-sobolev-equivalence}
  \norm{f}_{H^r(\Sd)}^2
  \asymp
  \sum_{k\geq0}2^{2kr}\norm{F_k}_{L^2(\Sd)}^2.
 \end{equation}
 Choose $k_0$ so that $2^{k_0}\leq n<2^{k_0+1}$.  For $k<k_0$, the degree
 of $F_k$ is at most $2^{k+1}\leq 2^{k_0}\leq n$, so the exactness in
 \eqref{eq:comparable-measure-exactness}
 gives $\mathcal L_m(F_k)=0$ for $k<k_0$. 
 We then estimate the bands $k\geq k_0$.  Denoting $L_k:=2^{k+1}$, we have
 $F_k\in\Pi_{L_k}(\Sd)$.
 Since $q(\boldsymbol{B}_m)\geq c_dm^{-1/d}\asymp n^{-1}$ and
$L_k\geq2^{k_0+1}>n$, we have
 $q(\boldsymbol{B}_m)\geq c_*L_k^{-1}$ for a constant $c_*>0$ independent
 of $k$, $m$, and $w$.  Fix $0<a\leq c_*$, sufficiently small for the mesh
 condition below.  Then $\boldsymbol{B}_m$ is $aL_k^{-1}$-separated.
 Extend it to a maximal $aL_k^{-1}$-separated set
 $\boldsymbol{C}_k$ (so no further point can be added), and let
 $\{R_\xi:\xi\in\boldsymbol{C}_k\}$ be its spherical Voronoi decomposition.
 Separation and maximality give, respectively,
 \begin{equation*}
  B\left(\xi,\frac{a}{2L_k}\right)\subset R_\xi
  \subset B\left(\xi,\frac{a}{L_k}\right),
  \qquad
  \sigma_d(R_\xi)\asymp L_k^{-d}.
 \end{equation*}
 Our choice of $a$ ensures the mesh condition in the
 Marcinkiewicz--Zygmund inequality of Mhaskar, Narcowich, and Ward
 \cite[Theorem~3.1]{mhaskar2001spherical}, which implies
 \begin{equation*}
  \sum_{\xi\in\boldsymbol{C}_k}
  \sigma_d(R_\xi)|F_k(\xi)|^2
  \lesssim \norm{F_k}_{L^2(\Sd)}^2.
 \end{equation*}
 Since $\boldsymbol{B}_m\subset\boldsymbol{C}_k$ and
 $\sigma_d(R_\xi)\gtrsim L_k^{-d}$, it follows that
 \begin{equation*}
  \sum_{x\in\boldsymbol{B}_m}|F_k(x)|^2
  \lesssim L_k^d\norm{F_k}_{L^2(\Sd)}^2.
 \end{equation*}
 Finally, $L_k\asymp2^k$ and $m\asymp n^d$ lead to, for $k\geq k_0$,
 \begin{equation}\label{eq:high-frequency-sampling}
  \frac1m\sum_{x\in\boldsymbol{B}_m}|F_k(x)|^2
  \lesssim
  \left(\frac{2^k}{n}\right)^d
  \norm{F_k}_{L^2(\Sd)}^2.
 \end{equation}
 Since $2^k/n>1/2$ for $k\geq k_0$, Cauchy--Schwarz, the uniform bound on
 $w$, and \eqref{eq:high-frequency-sampling} give
 \begin{equation*}
  \abs{\mathcal L_m(F_k)}
  \leq
  \left(\frac1m\sum_{x\in\boldsymbol{B}_m}|F_k(x)|^2\right)^{1/2}
  +\norm{w}_{L^2(\Sd)}\norm{F_k}_{L^2(\Sd)} \lesssim
  \left(\frac{2^k}{n}\right)^{d/2}
  \norm{F_k}_{L^2(\Sd)}.
 \end{equation*}
 Summing the high-frequency bands and noting
 $2^{kd/2}=2^{-k(r-d/2)}2^{kr}$, we formally obtain, for $r>d/2$,
 \begin{equation*}
 \begin{split}
  \abs{\mathcal L_m(f)}
  &\lesssim
  n^{-d/2}\sum_{k\geq k_0}
  2^{kd/2}\norm{F_k}_{L^2(\Sd)}\\
  &\leq n^{-d/2}
  \left(\sum_{k\geq k_0}2^{-2k(r-d/2)}\right)^{1/2}
  \left(\sum_{k\geq k_0}2^{2kr}
  \norm{F_k}_{L^2(\Sd)}^2\right)^{1/2}.
  \end{split}
 \end{equation*}
 Set $\alpha:=r-d/2>0$.  Then
 \begin{equation*}
  \left(\sum_{k\geq k_0}2^{-2k(r-d/2)}\right)^{1/2}=\left(2^{-2k_0\alpha}
    \sum_{j=0}^\infty2^{-2j\alpha}\right)^{1/2} =\frac{2^{-k_0\alpha}}{(1-2^{-2\alpha})^{1/2}}
  \lesssim_{r,d} 2^{-k_0(r-d/2)}.
 \end{equation*}
 Moreover, if \(r\in[r_0,S]\subset(d/2,\infty)\), then $\bigl(1-2^{-2\alpha}\bigr)^{-1/2}
\le
\bigl(1-2^{-2(r_0-d/2)}\bigr)^{-1/2}$,
so all constants above may be chosen uniformly for
\(r\in[r_0,S]\).
Together with \eqref{eq:dyadic-sobolev-equivalence} and $2^{k_0}\asymp n$, we
 therefore obtain
 \begin{equation*}
  \abs{\mathcal L_m(f)}
  \lesssim n^{-d/2}2^{-k_0(r-d/2)}
  \norm{f}_{H^r(\Sd)}\asymp n^{-d/2}n^{-(r-d/2)}
  \norm{f}_{H^r(\Sd)}
  =n^{-r}\norm{f}_{H^r(\Sd)}.
 \end{equation*}
 Finally, $n\asymp m^{1/d}$ gives $n^{-r}\asymp m^{-r/d}$ and proves
 \eqref{eq:comparable-measure-qmc}.
\end{proof}

The next lemma spectrally decomposes the discrepancy into a bounded
low-frequency polynomial and a controlled high-frequency remainder, with
the former governed by the number $K$ of separated blocks rather than $N$.

\begin{lemma}[Filtered discrepancy estimate]
\label{lem:filtered-separated-blocks}
 Let
 $\boldsymbol{X}_N=\mathop{\dot\bigcup}_{\nu=1}^K\boldsymbol{A}_\nu$
 be a $c$-separated $K$-block decomposition and $H:=\sum_{\nu=1}^K|\boldsymbol{A}_\nu|^{1/2}$. Suppose that $L^d\geq N$. Then,
there is a spherical polynomial $g_L$ of degree at most $2L$, with mean
 zero, such that
 \begin{equation}\label{eq:filtered-block-supnorm}
  \norm{g_L}_{L^\infty(\Sd)}\lesssim KL^d.
 \end{equation}
Moreover, for every
 $r>d/2$,
 \begin{equation}\label{eq:filtered-block-tail}
  \norm{\mathcal D_{\boldsymbol{X}_N}-g_L\sigma_d}_{H^{-r}(\Sd)}
  \lesssim H L^{d/2-r}.
 \end{equation}
The constants depend on \(c\), \(r\), and \(d\), but not on
\(N\), \(K\), or \(L\), and may be chosen uniformly for \(r\)
in any compact subset of \((d/2,\infty)\).
\end{lemma}

\begin{proof}
 Let $h\in C^\infty([0,\infty))$ satisfy $h=1$ on $[0,1]$ and $h=0$ on
 $[2,\infty)$, and define the filtered kernel
 \begin{equation*}
  \Phi_L(x,z):=
  \sum_{\ell=0}^\infty h(\ell/L)
  \sum_{k=1}^{Z(d,\ell)}Y_{\ell,k}(x)Y_{\ell,k}(z).
 \end{equation*}
 We define
 \begin{equation*}
  g_L(z):=\sum_{x\in\boldsymbol{X}_N}\Phi_L(x,z)-N.
 \end{equation*}
 Then $g_L$ is a spherical polynomial of degree at most $2L$.
 To match the shifted multiplier $(\ell+\gamma_d)/L$ in
 \cite[Theorem~3.5]{NarcowichPetrushevWard2006}, we denote by
 $\gamma_d:=(d-1)/2$ and, for $L\geq1$, define the even function
 \begin{equation*}
  \kappa_L(u):=
  \begin{cases}
   h(|u|-\gamma_d/L),& |u|\geq\gamma_d/L,\\
   1,& |u|<\gamma_d/L.
  \end{cases}
 \end{equation*}
 Since $h$ is identically one near the origin, the family
 $\{\kappa_L\}_{L\geq1}$ is smooth with uniformly bounded supports and
 derivatives, and
 $\kappa_L({(\ell+\gamma_d)}/{L})=h(\ell/L)$ for
 $\ell\geq0$.
 The cited localization theorem therefore gives, uniformly in $L$ and for
 every $M>d$,
 \begin{equation}\label{eq:filtered-kernel-localization}
  |\Phi_L(x,z)|\lesssim_M
  L^d(1+L\operatorname{dist}(x,z))^{-M}.
 \end{equation}
 Since $L^d\geq N$, condition \eqref{eq:block-separation} implies
 $q(\boldsymbol{A}_\nu)\geq c/L$.  Fix $z\in\Sd$ and divide
 $\boldsymbol{A}_\nu$ into sets $E_0:=\boldsymbol{A}_\nu\cap B(z,2L^{-1})$
 and $E_j:=\left\{x\in\boldsymbol{A}_\nu:
  2^jL^{-1}\leq\operatorname{dist}(x,z)<2^{j+1}L^{-1}\right\}$
 for $j\geq1$.
 Since $q(\boldsymbol{A}_\nu)\geq c/L$, the balls
 $B(x,c/(2L))$, $x\in\boldsymbol{A}_\nu$, are pairwise disjoint.  For
 $x\in E_0$, such a ball is contained in
 $B(z,(2+c/2)L^{-1})$; for $x\in E_j$, $j\geq1$, it is contained in
 $B(z,C_c2^jL^{-1})$.  The volume growth of spherical caps and disjointness
 therefore give $|E_0|L^{-d}\lesssim L^{-d}$
 and $|E_j|L^{-d}\lesssim (2^jL^{-1})^d$
 for $j\geq1$.  Hence
 \begin{equation}\label{eq:annular-packing-count}
  |E_0|\lesssim1,
  \qquad |E_j|\lesssim 2^{jd},\quad j\geq1.
 \end{equation}
 For $x\in E_0$, \eqref{eq:filtered-kernel-localization} gives
 $|\Phi_L(x,z)|\lesssim L^d$.  If $x\in E_j$ with $j\geq1$, then
 $L\operatorname{dist}(x,z)\geq2^j$, and hence
 $|\Phi_L(x,z)|\lesssim L^d2^{-jM}$.  Combining these bounds with the
 estimates for $|E_j|$ gives
 \begin{equation*}
  \sum_{x\in\boldsymbol{A}_\nu}|\Phi_L(x,z)|
  \lesssim L^d\left(1+\sum_{j\geq1}2^{jd}2^{-jM}\right)
  \lesssim L^d,
 \end{equation*}
 uniformly in $z$ and $\nu$, where the last series converges because $M>d$.
 Summing over the blocks and using $N\leq L^d\leq KL^d$ proves
 \eqref{eq:filtered-block-supnorm}.  Since $h(0)=1$, the constant harmonic
 in $\Phi_L$ gives $\int_{\Sd}\Phi_L(x,z)\,\dd\sigma_d(z)=1$, further yielding
 \begin{equation*}
  \int_{\Sd}g_L\,\dd\sigma_d
  =\sum_{x\in\boldsymbol{X}_N}
    \int_{\Sd}\Phi_L(x,z)\,\dd\sigma_d(z)-N
  =N-N=0.
 \end{equation*}

 For the second estimate,  we  define
 \begin{equation*}
  g_{L,\nu}:=
  \sum_{x\in\boldsymbol{A}_\nu}\Phi_L(x,\cdot)
  -|\boldsymbol{A}_\nu|.
 \end{equation*}
 Both $\mathcal D_{\boldsymbol{A}_\nu}$ and $g_{L,\nu}$ have mean zero. For
 $\ell\geq1$, we have
 \begin{equation*}
  \left\langle
   \mathcal D_{\boldsymbol{A}_\nu}-g_{L,\nu}\sigma_d,
   Y_{\ell,k}
  \right\rangle
  =(1-h(\ell/L))
   \sum_{x\in\boldsymbol{A}_\nu}Y_{\ell,k}(x).
 \end{equation*}
 Substituting these coefficients into
 \eqref{eq:negative-sobolev-norm} gives
 \begin{align*}
\norm{\mathcal D_{\boldsymbol{A}_\nu}
        -g_{L,\nu}\sigma_d}_{H^{-r}(\Sd)}^2 
  &=
  \sum_{\ell=1}^{\infty}(1+\lambda_\ell)^{-r}
  (1-h(\ell/L))^2
  \sum_{k=1}^{Z(d,\ell)}
  \left|\sum_{x\in\boldsymbol{A}_\nu}Y_{\ell,k}(x)\right|^2\\
  &=
  \sum_{x,y\in\boldsymbol{A}_\nu}
  \sum_{\ell=1}^{\infty}(1+\lambda_\ell)^{-r}
  (1-h(\ell/L))^2
  \sum_{k=1}^{Z(d,\ell)}Y_{\ell,k}(x)Y_{\ell,k}(y),
   \end{align*}
where we introduce the tail kernel
 \begin{equation*}
  \Psi_{L,r}(x,y):=
  \sum_{\ell=1}^\infty
  (1-h(\ell/L))^2(1+\lambda_\ell)^{-r}
  \sum_{k=1}^{Z(d,\ell)}Y_{\ell,k}(x)Y_{\ell,k}(y).
 \end{equation*}
 Choose a smooth partition of unity $\{\chi_j\}_{j\geq0}$ on
 $[L,\infty)$ subordinate to the dyadic annuli of scale $R_j:=2^jL$, with
 $\operatorname{supp}\chi_j\subset[c_1R_j,c_2R_j]$ 
 and $|\chi_j^{(q)}(t)|\leq C_qR_j^{-q}$ for $q\geq0$,
 with constants independent of $j$ and $L$.  Define the $j$th band
 multiplier on $[0,\infty)$ by
 \begin{equation*}
  a_{j,L,r}(t):=
  \chi_j(t)(1-h(t/L))^2
  (1+t(t+d-1))^{-r}.
 \end{equation*}
 On its support, $t\asymp R_j$.  The factor $1-h(t/L)$ is constant outside
 finitely many initial bands, and $R_j/L$ is bounded on those bands.
 The product rule therefore gives, uniformly in $j$ and $L$,
 \begin{equation}\label{eq:tail-band-symbol-bounds}
  |a_{j,L,r}^{(q)}(t)|\lesssim_{q,r}R_j^{-2r-q},
  \qquad q\geq0.
 \end{equation}
 After extending $a_{j,L,r}$ by zero to the negative half-line, let
 \[\kappa_{j,L,r}(u):=R_j^{2r}a_{j,L,r}(R_j|u|-\gamma_d).\] The derivative bounds \eqref{eq:tail-band-symbol-bounds} show that these
 even multipliers have
 uniformly bounded supports and derivatives, and
 \begin{equation*}
  a_{j,L,r}(\ell)
  =R_j^{-2r}\kappa_{j,L,r}
   \left(\frac{\ell+\gamma_d}{R_j}\right).
 \end{equation*}
 The localization theorem \cite[Theorem~3.5]{NarcowichPetrushevWard2006}
 contributes the kernel scale $R_j^d$, while the prefactor above contributes
 $R_j^{-2r}$.  If $\Psi_{j,L,r}$ denotes the corresponding band kernel,
 then
 $|\Psi_{j,L,r}(x,y)|\lesssim
  R_j^{d-2r}
  (1+R_j\operatorname{dist}(x,y))^{-M}$.
 Summing the bands gives, for every $M>d$,
 \begin{equation}\label{eq:tail-kernel-localization}
  |\Psi_{L,r}(x,y)|
  \lesssim
  \sum_{j=0}^\infty
  R_j^{d-2r}
  (1+R_j\operatorname{dist}(x,y))^{-M}.
 \end{equation}
 Since $q(\boldsymbol{A}_\nu)\geq c/L\geq c/R_j$, the annular packing
 argument leading to \eqref{eq:annular-packing-count}, now applied at scale
 $R_j^{-1}$, gives uniformly in
 $x\in\boldsymbol{A}_\nu$,
 \begin{equation*}
  \sum_{y\in\boldsymbol{A}_\nu}
  (1+R_j\operatorname{dist}(x,y))^{-M}\lesssim1.
 \end{equation*}
 Consequently, since $R_j=2^jL$ and $r>d/2$, so that $d-2r<0$,
 \begin{equation*}
  \norm{\mathcal D_{\boldsymbol{A}_\nu}-g_{L,\nu}\sigma_d}_{H^{-r}(\Sd)}^2
  \lesssim |\boldsymbol{A}_\nu|
  \sum_{j=0}^\infty R_j^{d-2r} =|\boldsymbol{A}_\nu|L^{d-2r}
  \sum_{j=0}^\infty 2^{j(d-2r)}
  \lesssim_{d,r}|\boldsymbol{A}_\nu|L^{d-2r}.
 \end{equation*}
 Finally, since
 \begin{equation*}
  \mathcal D_{\boldsymbol{X}_N}-g_L\sigma_d
  =\sum_{\nu=1}^K
   (\mathcal D_{\boldsymbol{A}_\nu}-g_{L,\nu}\sigma_d),
 \end{equation*}
 the triangle inequality yields
 \eqref{eq:filtered-block-tail}.
\end{proof}

\begin{theorem}[Completion controlled by block complexity]
\label{thm:block-complexity-completion}
 Fix $d/2<r_0\leq S$ with $S\geq d$.  Suppose that the $N$-point set $\boldsymbol{X}_N$ admits
 a $c$-separated $K$-block decomposition
$\boldsymbol{X}_N=\mathop{\dot\bigcup}_{\nu=1}^K
 \boldsymbol{A}_\nu$, and let
 $H:=\sum_{\nu=1}^K|\boldsymbol{A}_\nu|^{1/2}$.  Then $\boldsymbol{X}_N$ is contained in
 an $N'$-point set $\boldsymbol{X}_{N'}$ such that, simultaneously for every
 $r\in[r_0,S]$,
 \begin{equation*}
  \sup_{\norm{f}_{H^r(\Sd)}\leq1}
  \abs{
   \frac1{N'}\sum_{x\in \boldsymbol{X}_{N'}}f(x)-\int_{\Sd}f\,\dd\sigma_d
  }
  \lesssim (N')^{-r/d},
 \end{equation*}
 and
 \begin{equation}\label{eq:block-completion-cardinality}
 N'\lesssim H^2K^{2S/d-1}.
 \end{equation}
 Moreover, there is a constant $c_*=c_*(d)>0$ such that
 $\boldsymbol{X}_{N'}$ admits a $\min\{c,c_*\}$-separated
 $(K+1)$-block decomposition.  In particular, the conclusion can be
 iterated with a fixed separation constant.
\end{theorem}

\begin{proof}
 Let
  \begin{equation}\label{eq:block-choice-L-m}
  L:=\left\lceil
   \bigl(H^2K^{2S/d-2}\bigr)^{1/d}
  \right\rceil,
  \qquad
  m:=\left\lceil A_{d,c}KL^d\right\rceil,
 \end{equation}
 where $A_{d,c}$ will be fixed below. Since $H^2\geq N$,
 $S\geq d$, and $K\geq1$, we have
 $N\leq L^d\asymp H^2K^{2S/d-2}$ and
 $m\asymp H^2K^{2S/d-1}$. In particular, \(m\gtrsim N\).
 Let $g_L$ be supplied by Lemma~\ref{lem:filtered-separated-blocks}, and set
 \[w_L:=1-\frac{g_L}{m}.\]  Let $C_{d,c}$ be the implicit constant in
 \eqref{eq:filtered-block-supnorm}, so that
 $\|g_L\|_{L^\infty(\Sd)}\leq C_{d,c}KL^d$, and let $m_0$ be the uniform
 threshold in Lemma~\ref{lem:comparable-measure-block}.  Choose
 $A_{d,c}\geq\max\{2C_{d,c},m_0\}$.  Then $m\geq m_0$ and
 \begin{equation*}
  \left\|\frac{g_L}{m}\right\|_{L^\infty(\Sd)}
  \leq\frac{C_{d,c}KL^d}{A_{d,c}KL^d}\leq\frac12.
 \end{equation*}
 Since $g_L$ has mean zero, it follows that
  $\int_{\Sd}w_L\,\dd\sigma_d=1$ and 
  $1/2\leq w_L\leq3/2$.
 By the dual characterization of $H^{-r}(\Sd)$, together with $mw_L\sigma_d=m\sigma_d-g_L\sigma_d$,
 Lemma~\ref{lem:comparable-measure-block} gives an $m$-point set
 $\boldsymbol{B}_m$ such that for any $r\in[r_0,S]$,
 \begin{equation}\label{eq:block-new-residual}
  \norm{\mathcal D_{\boldsymbol{B}_m}+g_L\sigma_d}_{H^{-r}(\Sd)}=
  \left\|
   \sum_{x\in\boldsymbol{B}_m}\delta_x-mw_L\sigma_d
  \right\|_{H^{-r}(\Sd)}
  \lesssim m^{1-r/d},
 \end{equation}
 and
 \begin{equation*}
  q(\boldsymbol{B}_m)\geq b_dm^{-1/d}
 \end{equation*}
 for a constant $b_d>0$ independent of $N$, $K$, and $c$.

We next compare the high-frequency remainder of the old discrepancy with
 the error scale of the new block.  For $r\leq S$, the
 choices in \eqref{eq:block-choice-L-m} give
 \begin{equation}\label{eq:block-tail-balance}
  H L^{d/2-r}
  \asymp (KL^d)^{1-r/d}K^{(r-S)/d}
  \asymp m^{1-r/d}K^{(r-S)/d}
  \leq m^{1-r/d}.
 \end{equation}
 
 We may assume that $\boldsymbol{B}_m$ is disjoint from
 $\boldsymbol{X}_N$.  Indeed, otherwise perturb the coincident nodes within
 pairwise disjoint neighborhoods, with each displacement smaller than
 $b_dm^{-1/d}/4$.  The perturbed block $\widetilde{\boldsymbol{B}}_m$ then
 avoids $\boldsymbol{X}_N$ and has separation at least
 $(b_d/2)m^{-1/d}$.  Since $r_0>d/2$, the map
 $x\mapsto\delta_x$ is continuous from $\Sd$ into $H^{-r_0}(\Sd)$, so the
 perturbations can be chosen such that, for every $r\in[r_0,S]$,
 \begin{equation*}
  \norm{\mathcal D_{\widetilde{\boldsymbol{B}}_m}
        -\mathcal D_{\boldsymbol{B}_m}}_{H^{-r}(\Sd)}
  \lesssim
  \norm{\mathcal D_{\widetilde{\boldsymbol{B}}_m}
        -\mathcal D_{\boldsymbol{B}_m}}_{H^{-r_0}(\Sd)}
  \leq m^{1-S/d}\leq m^{1-r/d}.
 \end{equation*}
 Thus the perturbation is absorbed into \eqref{eq:block-new-residual}, and
 we relabel $\widetilde{\boldsymbol{B}}_m$ as $\boldsymbol{B}_m$.  For
 $\boldsymbol{X}_{N'}:=\boldsymbol{X}_N\cup\boldsymbol{B}_m$
 and $N'=N+m\asymp m$, we have
 \begin{equation*}
  \mathcal D_{\boldsymbol{X}_{N'}}
  =(\mathcal D_{\boldsymbol{X}_N}-g_L\sigma_d)
   +(\mathcal D_{\boldsymbol{B}_m}+g_L\sigma_d).
 \end{equation*}
 Together with estimates \eqref{eq:filtered-block-tail},
 \eqref{eq:block-new-residual}, and \eqref{eq:block-tail-balance}, we have $\norm{\mathcal D_{\boldsymbol{X}_{N'}}}_{H^{-r}(\Sd)}
  \lesssim m^{1-r/d}$ for $r\in[r_0,S]$.
 Since $N'\asymp m$, division by $N'$ yields the asserted QMC error
 \begin{equation*}
  \sup_{\norm{f}_{H^r(\Sd)}\leq1}
  \abs{
   \frac1{N'}\sum_{x\in\boldsymbol{X}_{N'}}f(x)
   -\int_{\Sd}f\,\dd\sigma_d
  }
  =\frac1{N'}
   \norm{\mathcal D_{\boldsymbol{X}_{N'}}}_{H^{-r}(\Sd)}
  \lesssim (N')^{-r/d}.
 \end{equation*}
Since $m\gtrsim N$, we have
\[
 N'=N+m\asymp m\asymp H^2K^{2S/d-1}.
\]
This proves \eqref{eq:block-completion-cardinality}.

 Finally, let $\boldsymbol{A}'_\nu:=\boldsymbol{A}_\nu$, $1\leq\nu\leq K$, and
  $\boldsymbol{A}'_{K+1}:=\boldsymbol{B}_m$.
 Since $\boldsymbol{B}_m$ is disjoint from $\boldsymbol{X}_N$, these sets
 form the disjoint decomposition
 \begin{equation*}
  \boldsymbol{X}_{N'}
  =\mathop{\dot\bigcup}_{\nu=1}^{K+1}\boldsymbol{A}'_\nu.
 \end{equation*}
 Moreover, $N'\geq N$ and $N'\geq m$, so the old and new blocks satisfy
 \begin{equation*}
  q(\boldsymbol{A}_\nu)\geq cN^{-1/d}\geq c(N')^{-1/d},
  \qquad
  q(\boldsymbol{B}_m)\geq\frac{b_d}{2}m^{-1/d}
  \geq\frac{b_d}{2}(N')^{-1/d}.
 \end{equation*}
 Hence $\boldsymbol{X}_{N'}$ admits a
 $\min\{c,b_d/2\}$-separated $(K+1)$-block decomposition.  Taking
 $c_*:=b_d/2$ proves the final assertion and keeps the separation constant
 fixed under iteration.
\end{proof}

\begin{remark}[Application in the subcritical range]
\label{rem:sec3subcritical}
For $d/2<s<d$, taking $r_0=s$ and $S=d$ and using the
 block-size induction below gives $N_{j+1}\lesssim(j+1)N_j$.  This is weaker
 than Theorem~\ref{thm:subcritical-block}, which allows
 $N_{j+1}/N_j\to\rho$ for any $\rho>1$; completion is therefore needed only
 for $s\geq d$.
\end{remark}

\subsection{Nested QMC designs}

We now iterate Theorem~\ref{thm:block-complexity-completion}.  Its
preservation of the separated-block decomposition controls the size of each
successive completion.

\begin{theorem}[Nested QMC designs for $s\geq d$]
\label{thm:critical-supercritical}
Fix $s\geq d$.  There is a nested sequence
 $\boldsymbol{X}_{N_0}\subset\boldsymbol{X}_{N_1}\subset\cdots$ whose
 successive cardinalities satisfy
 \begin{equation}\label{eq:critical-supercritical-growth}
  N_{j+1}\lesssim (j+1)^{2s/d-1}N_j.
 \end{equation}
 Every equal-weight rule in this sequence satisfies, for $j\geq0$,
 \begin{equation}\label{eq:critical-supercritical-nested-error}
  \abs{
   \frac1{N_j}\sum_{x\in\boldsymbol{X}_{N_j}}f(x)
   -\int_{\Sd}f\,\dd\sigma_d
  }
  \lesssim N_j^{-s/d}\norm{f}_{H^s(\Sd)},
  \qquad f\in H^s(\Sd).
 \end{equation}
 Moreover, every prescribed $N$-point set $\boldsymbol{X}_N$ is contained
 in an $N'$-point set $\boldsymbol{X}_{N'}$ satisfying
 \begin{equation*}
  N'\lesssim N^{1+2s/d}
 \end{equation*}
 and
 \begin{equation*}
  \abs{
   \frac1{N'}\sum_{x\in\boldsymbol{X}_{N'}}f(x)
   -\int_{\Sd}f\,\dd\sigma_d
  }
  \lesssim (N')^{-s/d}\norm{f}_{H^s(\Sd)},
  \qquad f\in H^s(\Sd).
 \end{equation*}
\end{theorem}

\begin{proof}
We start with a finite separated set \(\boldsymbol{X}_{N_0}\), regarded as a
single block.  It admits a \(c\)-separated \(1\)-block decomposition, and its
error at this fixed initial level is finite and may be absorbed into the
implicit constant in \eqref{eq:critical-supercritical-nested-error}.
Suppose inductively that $\boldsymbol{X}_{N_j}$ has a uniformly separated
decomposition $\mathop{\dot\bigcup}_{\nu=1}^{j+1}\boldsymbol{A}_{j,\nu}$,
and put $H_j:=\sum_{\nu=1}^{j+1}|\boldsymbol{A}_{j,\nu}|^{1/2}$.  We also
maintain $H_j^2\lesssim N_j$, which holds for the initial one-block
decomposition.  The completion in Theorem~\ref{thm:block-complexity-completion},
with $r_0=S=s$ and $K=j+1$, adds a block of size
$m_j\asymp H_j^2(j+1)^{2s/d-1}$.
Since $2s/d-1\geq1$ and $H_j^2\geq N_j$, we have
$N_{j+1}=N_j+m_j\asymp m_j$.  Moreover,
\[
 H_{j+1}^2=(H_j+m_j^{1/2})^2
 \leq2H_j^2+2m_j\lesssim N_{j+1},
\]
so the induction continues and
$N_{j+1}\lesssim N_j(j+1)^{2s/d-1}$.  The same theorem gives
$\mathcal E_s(\boldsymbol{X}_{N_{j+1}})\lesssim N_{j+1}^{-s/d}$, proving
\eqref{eq:critical-supercritical-growth} and \eqref{eq:critical-supercritical-nested-error}.

For the one-step statement, decompose an arbitrary prescribed set
\(\boldsymbol{X}_N\) into its \(N\) singleton blocks.  Applying
Theorem~\ref{thm:block-complexity-completion} with $r_0=S=s$, $K=N$, and
$H=N$ gives $N'\lesssim H^2K^{2s/d-1}=N^{1+2s/d}$. Together with $\mathcal E_s(\boldsymbol{X}_{N'})
 \lesssim (N')^{-s/d}$, it proves the one-step completion assertion.
\end{proof}

\section{Concluding remarks}
\label{sec:concluding-remarks}

We have shown that, for every fixed $s>d/2$, the optimal Sobolev cubature
rate is compatible with equal-weight nesting of cubature points on $\Sd$.  The construction
changes at the critical index $s=d$.  Below this index, geometrically growing
QMC blocks can be combined directly.  At and above it, the discrepancy of
the inherited nodes must instead be compensated by the newly added block.
The resulting completion argument provides nested QMC designs throughout the
full range $s>d/2$.

Our results establish an optimal-rate QMC counterpart to the nested
spherical $t$-designs studied by Zheng and Zhuang in 
\cite{ZhengZhuang2024Nested}.  Starting from an optimal-order spherical
$t_1$-design with $t_1<t$, their general construction produces a spherical
$t$-design with a total of $P=O(t^{2d+1})$ points.  Polynomial exactness and
\eqref{eq:designs-are-qmc} then give, for every fixed $s>d/2$,
$\mathcal E_s(\boldsymbol{X}_P)\lesssim t^{-s}\lesssim
 P^{-s/(2d+1)}$.
Thus their currently proved general cardinality bound yields the Sobolev
exponent $s/(2d+1)$ when expressed in terms of the total number of completed nodes.  For $s\ge d$, Theorem~\ref{thm:critical-supercritical} replaces polynomial exactness by the QMC condition and attains the optimal exponent $s/d$:
$\mathcal E_s(\boldsymbol{X}_{N'})\lesssim (N')^{-s/d}$.
The bounds $P=O(t^{2d+1})$ in \cite{ZhengZhuang2024Nested} and $N'\lesssim N^{1+2s/d}$ in the present work
 describe different parameters. The former is expressed in the target degree $t$, while the
latter depends on the inherited cardinality $N$ and the Sobolev index $s$.
The comparison therefore concerns the integration error rate, rather than a direct comparison of the two completion sizes.
Moreover, polynomial exactness is stronger than the QMC condition, and
$P=O(t^{2d+1})$ is an upper bound for the known construction.  When
$t=mt_1$ for a fixed rational number $m>1$, Zheng and Zhuang obtain the exact
optimal order $O(t^d)$ for suitably chosen inherited designs, with constants
depending on $m$. The uniform $O(t^d)$ exact-completion problem remains open.
Our result shows that, when the objective is optimal Sobolev integration,
arbitrary-node completion can retain the optimal QMC exponent.
A natural next question is the cost of refinement.  Below
$s=d$, our construction allows $N_{j+1}/N_j\to\rho$ for every $\rho>1$.
For $s\geq d$, it gives
$N_{j+1}\lesssim(j+1)^{2s/d-1}N_j$ and the one-step bound
$N'\lesssim N^{1+2s/d}$.  Whether these can be reduced to thin refinement remains open.

The setting of non-equal weights, which is also natural in spherical
quadrature and MZ discretization
\cite{filbir2011marcinkiewicz,Mhaskar2006Weighted}, offers a further direction.
Sloan and Womersley recently considered non-equal-weight QMC designs
in \cite{SloanWomersley2026}.  If arbitrary positive
weights are allowed, a trivial construction of nested QMC designs is immediate: the inherited nodes
may simply be assigned a vanishing total mass.  A substantive weighted theory
should therefore incorporate balanced weights and quantitative inheritance
between successive levels.  The probabilistic results of Zheng and Zhuang \cite{ZhengZhuang2026} concern exact quadrature on growing uniformly sampled point sets, with weights recomputed at successive levels.  Combining these ideas with
stable weight inheritance may lead to weighted nested QMC rules with smaller
completion costs.

{\small
\bibliographystyle{siamplain}
\bibliography{myref}
}

\end{document}